\documentclass[12pt]{article} 

\usepackage[utf8]{inputenc} 
\usepackage[all]{xy}
\usepackage{epstopdf}

\usepackage{geometry} 
\usepackage{graphicx} 

\usepackage{booktabs} 
\usepackage{array} 
\usepackage{paralist} 
\usepackage{verbatim} 
\usepackage{subfig} 

\usepackage{amsmath}
\usepackage{amsthm}
\usepackage{amsfonts}
\usepackage{amsopn}
\usepackage{amssymb}

\newtheorem{thm}{Theorem}[section]
\newtheorem{cor}[thm]{Corollary}
\newtheorem{lem}[thm]{Lemma}
\newtheorem{prop}[thm]{Proposition}

\theoremstyle{remark}
\newtheorem{rem}[thm]{Remark}
\newtheorem{example}[thm]{Example}

\newcommand{\cO}{{\cal O}}
\newcommand{\ra}{\rightarrow}
\newcommand{\CC}{{\mathbb C}}
\newcommand{\FF}{{\mathbb F}}
\newcommand{\NN}{{\mathbb N}}
\newcommand{\PP}{{\mathbb P}}
\newcommand{\ZZ}{{\mathbb Z}}

\DeclareMathOperator{\ord}{ord}

\DeclareMathOperator{\Aut}{Aut}

\title{Faithfulness of Actions on Riemann-Roch Spaces}
\author{Bernhard K\"ock and Joseph Tait}

\begin{document}
\maketitle

\begin{quote}
  {\bf Abstract} Given a faithful action of a finite group $G$ on an algebraic curve~$X$ of genus $g_X\geq 2$, we give explicit criteria for the induced action of~$G$ on the Riemann-Roch space~$H^0(X,\cO_X(D))$ to be faithful, where $D$ is a $G$-invariant divisor on $X$ of degree at least~$2g_X-2$. This leads to a concise answer to the question when the action of~$G$ on the space~$H^0(X, \Omega_X^{\otimes m})$ of global holomorphic polydifferentials of order $m$ is faithful. If $X$ is hyperelliptic, we furthermore provide an explicit basis of~$H^0(X, \Omega_X^{\otimes m})$. Finally, we give applications in deformation theory and in coding theory and we discuss the analogous problem for the action of~$G$ on the first homology $H_1(X, \ZZ/m\ZZ)$ if $X$ is a Riemann surface.
\end{quote}

  \section{Introduction}

  Let $X$ be a connected smooth projective algebraic curve over
an algebraically closed field $k$ equipped with a faithful
action of a finite group $G$ of order $n$. Furthermore, let $D= \sum_{P\in X} n_P [P]$ be a $G$-invariant divisor on~$X$. Then $G$ also acts on the Riemann-Roch space $H^0(X, \cO_X(D))$ consisting of all meromorphic functions on~$X$ whose order at any point $P \in X$ is at least $-n_P$.

A widely studied problem is to determine
the structure of $H^0(X,\cO_X(D))$ as a module over the group
ring $k[G]$. When $D$ is the canonical divisor and $k = \CC$, this amounts to calculating (the character of) the representation of~$G$ on the complex vector space~$H^0(X,\Omega_X)$ of global holomorphic differentials on the Riemann surface~$X$ and goes back to Chevalley-Weil \cite{CW}. If the canonical projection $\pi: X \rightarrow
Y$ from $X$ to the quotient curve $Y = X/G$ is tamely ramified,
fairly general and explicit answers to this problem have been found by Kani \cite{Kani86} and Nakajima \cite{Nak2}. In the case of arbitrary
wild ramification the explicit calculation of the
$k[G]$-isomorphism class of $H^0(X,\cO_X(D))$ is still an open
problem, but many partial and related results are known, see the recent papers \cite{cohogsheaves}, \cite{FWK}, \cite{FGMPW}, \cite{GJK}, \cite{Hor}, \cite{Karan12} and the literature cited there.

In this paper we look at the weaker question of whether the group~$G$ acts faithfully on $H^0(X,\cO_X(D))$. To this end, we first prove formulae for the dimension of the subspace $H^0(X,\cO_X(D))^G$ of $H^0(X,\cO_X(D))$ fixed by $G$, provided the degree of~$D$ is sufficiently large, see Proposition~\ref{dimD} and its corollaries.

In Sections~\ref{GlobalDifferentials} and~\ref{RiemannRochSpaces} we give explicit criteria for the action on $H^0(X, \cO_X(D))$ to be trivial and finally criteria for this action to be faithful if the degree of $D$ is at least $2g_X-2$. The latter criteria become particularly concise when $D$ is a positive multiple of the canonical divisor, see Theorem~\ref{faithful1} and Corollary~\ref{faithful2}, and can be summarized as follows.

Let $p \ge 0$ denote the characteristic of $k$ and let $g_X$ and $g_Y$ denote the genus of~$X$ and~$Y$, respectively. Furthermore, let $m\geq1$ and suppose that $g_X\geq 2$. We recall that a hyperelliptic involution of $X$ is an automorphism $\sigma$ of $X$ of order $2$ such that the quotient curve $X/\langle \sigma \rangle$ is isomorphic to $\mathbb{P}_k^1$. Then $G$ acts faithfully on the space $H^0(X,\Omega_X^{\otimes m})$ of global (poly)differentials of order $m$, unless $G$ contains a hyperelliptic involution and either $m=1$ and $p=2$ or $m=2$ and $g_X=2$.

If $X$ is a Riemann surface, versions of this result can also be found in Lewittes paper~\cite{Lew} or derived from Broughton's paper~\cite{Bro}. Furthermore, it is possible to give different and sometimes shorter proofs of parts of this result using deeper theorems about algebraic curves, see the relevant remarks in Sections~4 and~5.

  In Section~\ref{hyperelliptic} we look at the particular case when $X$ is hyperelliptic and give an explicit basis for the space~$H^0(X,\Omega_X^{\otimes m})$. This will yield a `hands-on' proof of the above result if $G$ is generated by the hyperelliptic involution.

  Faithful actions of permutation groups on Goppa codes play an important role in Coding Theory. In Section~\ref{Goppa} we apply Corollary~\ref{trivialD4} to obtain such actions.

  The dimension formula proved in Section~\ref{DimensionFormulae} moreover allows us to compute the dimension of the tangent space of the equivariant deformation functor associated with~$(G,X)$ provided the group~$G$ satisfies a certain assumption, see Theorem~\ref{deformation}. This theorem generalizes a main result in \cite{quaddiffequi} and considerably simplifies its proof.

  Finally, in Section~\ref{homology}, we investigate a striking analogy between faithful action on $H^0(X,\Omega_X^{\otimes m})$ and faithful action on the first homology~$H_1(X,\ZZ/m\ZZ)$ if $X$ is a Riemann surface.


In this final paragraph of the introduction we explain some notations and fundamental facts that we will use throughout the paper.
We write
  \[
      R=\sum_{P\in X}\delta_P[P]
  \]
for the ramification divisor of $\pi:X \rightarrow Y$.
The Hurwitz formula (see \cite[Ch.~IV, Corollary~2.4]{hart}) states that
  \begin{equation}\label{Hurwitz}
     2g_X-2 = n (2g_Y-2) + \textrm{deg}(R)
  \end{equation}
(where $n= \textrm{ord}(G)$). Furthermore, Hilbert's formula states that
  \begin{equation}\label{Hilbert}
     \delta_P=\sum_{j=0}^{\infty}(\ord(G_j(P))-1),
  \end{equation}
where $G_j(P)$ is the $j^\textrm{th}$ ramification group at $P$ in lower notation, see \cite[Ch.~IV, {\S}~1]{localfields}.
For any $P\in X$, let $e_P=\ord(G_0(P))$ denote the ramification index at $P$.
For any $Q\in Y$ we write $\delta_Q$ for $\delta_P$ and $e_Q$ for $e_P$ where $P\in \pi^{-1}(Q)$; recall that the cardinality of $\pi^{-1}(Q)$ is $\frac{n}{e_Q}$. As usual, the sheaf of differentials on $X$ is denoted by $\Omega_X$ and its $m^\textrm{th}$ tensor power by $\Omega_X^{\otimes m}$ for any $m\geq 2$.
Sections of $\Omega_X^{\otimes m}$ are called {\em polydifferentials of order $m$} and, if $m=2$, {\em quadratic differentials}.
We let $K_Y$ be a canonical divisor on $Y$.
Then the divisor $K_X:=\pi^*(K_Y)+R$ is a $G$-invariant canonical divisor on $X$ by \cite[{\S} IV, Prop. 2.3]{hart} and $\cO_X(mK_X)$ and $\Omega_X^{\otimes m}$ are isomorphic as $G$-sheaves.

  \section{Dimension Formulae}\label{DimensionFormulae}

In this section, given  a $G$-invariant divisor $D$ on our curve $X$ of sufficiently large degree, we are going to compute the dimension of the subspace $H^0(X,\cO_X(D))^G$ of the Riemann-Roch space $H^0(X,\cO_X(D))$ fixed by the action of the group~$G$. When $D$ is a multiple of the canonical divisor $K_X$ on $X$, we will in particular obtain a formula for the dimension of the space $H^0(X,\Omega_X^{\otimes m})^G$ of global $G$-invariant holomorphic polydifferentials of order~$m$.

%

%
%

We first introduce some notations. Let $D=\sum_{P\in X} n_P [P]$ be a $G$-invariant divisor on $X$ (i.e.\ $n_{\sigma(P)}=n_P$ for all $\sigma \in G$ and $P\in X$). For any $Q\in Y$, let $n_Q$ be equal to $n_P$ for any $P\in \pi^{-1}(Q)$. Let $\cO_X(D)$ denote the corresponding equivariant invertible $\cO_X$-module, as usual.
Furthermore let $\pi_*^G(\cO_X(D))$ denote the subsheaf of the direct image $\pi_*(\cO_X(D))$ fixed by the obvious action of $G$ on $\pi_*(\cO_X(D))$ and let $\left\lfloor \frac{\pi_*(D)}{n}
\right \rfloor$ denote the divisor on $Y$ obtained from the push-forward $\pi_*(D)$ by replacing the coefficient $m_Q$ of $Q$ in $\pi_*(D)$ with the integral part $\left \lfloor \frac{m_Q}{n} \right \rfloor$ of $\frac{m_Q}{n}$ for every $Q \in Y$.
The function fields of $X$ and $Y$ are denoted by $K(X)$ and $K(Y)$ respectively.
For any $P \in X$ and $Q \in Y$ let $\ord_P$ and $\ord_Q$ denote the respective valuations of $K(X)$ and $K(Y)$ at $P$ and $Q$.
Finally, let~$\langle a \rangle$ denote the fractional part of any $a\in \mathbb{R}$, i.e.\ $\langle a \rangle = a - \lfloor a \rfloor$.

The next (folklore) lemma is the main idea in the proof of our dimension formulae.

  \begin{lem}\label{folklore}
    Let $D=\sum_{P\in X} n_P [P]$ be a $G$-invariant divisor on $X$.
    Then the sheaves $\pi_*^G(\cO_X(D))$ and $\cO_Y\left(\left\lfloor \frac{\pi_*(D)}{n}\right \rfloor\right)$ are equal as subsheaves of the constant sheaf $K(Y)$ on $Y$.
    In particular the sheaf $\pi_*^G(\cO_X(D))$ is an invertible $\cO_Y$-module.
  \end{lem}
  \begin{proof}
    For every open subset $V$ of $Y$ we have
      \[
	 \pi_*^G(\cO_X(D))(V) = \cO_X(D) (\pi^{-1}(V))^G \subseteq K(X)^G = K(Y).
      \]
    In particular both sheaves are subsheaves of the constant sheaf $K(Y)$ as stated.
    It therefore suffices to check that their stalks are equal.
    For any $Q \in Y$ and $P \in \pi^{-1}(Q)$  we have
      \begin{eqnarray*}
	 \lefteqn{\pi_*^G\left(\cO_X(D)\right)_Q = \cO_X(D)_P \cap K(Y)}\\
	  &=& \left\{f \in K(Y): \textrm{ord}_P(f) \ge -n_P\right\}\\
	  &=& \left\{f \in K(Y): \textrm{ord}_Q(f) \ge - \frac{n_P}{e_P}\right\}\\
	  &=& \left\{ f \in K(Y): \textrm{ord}_Q(f) \ge - \left\lfloor\frac{n_P}{e_P} \right\rfloor \right\}\\
	  &=& \cO_Y\left(\left\lfloor \frac{\pi_*(D)}{n} \right\rfloor\right)_Q,
      \end{eqnarray*}
    as desired.
  \end{proof}

  The following proposition computes the dimension of the subspace~$H^0(X,\cO_X(D))^G $ of $H^0(X,\cO_X(D))$ fixed by~$G$.

  \begin{prop}\label{dimD}
  Let $D=\sum_{P \in X} n_P [P]$ be a $G$-invariant divisor on $X$ such that
  \[\deg(D) > 2g_X - 2 - \sum_{P \in X} \sum_{j \ge 1} \left(\ord(G_j(P)) - 1\right).\]
  Then we have:
  \[\dim_k H^0(X,\cO_X(D))^G = 1- g_Y + \frac{1}{n} \deg(D) - \sum_{Q \in Y} \left\langle \frac{n_Q}{e_Q} \right\rangle.\]
  \end{prop}

  \begin{rem}
  Note that the double sum $\sum_{P \in X} \sum_{j \ge 1}\left(\ord(G_j(P))-1\right)$ is non-ne\-ga\-tive and it is zero if and only if $\pi$ is at most tamely ramified. Subtracting this double sum makes the the usual bound $2g_X-2$ smaller and hence the statement stronger, see also the proof of the next corollary.
  \end{rem}

  \begin{proof}
  We have
  \begin{eqnarray*}
  \lefteqn{\deg \left\lfloor \frac{\pi_*(D)}{n}\right\rfloor = \sum_{Q \in Y} \left\lfloor \frac{n}{e_Q} \frac{n_Q}{n} \right\rfloor
  = \sum_{Q \in Y} \left\lfloor \frac{n_Q}{e_Q} \right\rfloor}\\
  & = & \sum_{Q \in Y} \left( \frac{n_Q}{e_Q} - \left\langle \frac{n_Q}{e_Q} \right\rangle \right) \\
  & \ge & \sum_{Q \in Y} \left(\frac{n_Q}{e_Q} - \frac{e_Q - 1}{e_Q} \right)\\
  &= & \sum_{P \in X} \left( \frac{n_P}{n} - \frac{e_P - 1}{n} \right) \\
  & = & \frac{1}{n}\left(\deg(D) - \sum_{P \in X} (e_P -1)\right)\\
  & > & \frac{1}{n} \left( 2g_X-2 - \sum_{P \in X} \sum_{j \ge 1} \left(\ord(G_j(P))-1 \right) - \sum_{P \in X} (e_P -1)\right)\\
  && \hspace*{5.9cm} \textrm{(by assumption)}\\
  & = & \frac{1}{n} \left( 2g_X -2 - \deg(R) \right) \hspace*{1.9cm} \textrm{(by Hilbert's formula~(\ref{Hilbert}))} \\
  & = & 2g_Y-2 \hspace*{4.5cm} \textrm{(by Hurwitz' formula~(\ref{Hurwitz}))}.
  \end{eqnarray*}
  Hence, using Lemma~\ref{folklore} and the Riemann-Roch formula \cite[Ch.~IV, {\S}1, Theorem~1.3 and Example~1.3.4]{hart}, we obtain
  \begin{eqnarray*}
  \lefteqn{\dim_k H^0(X,\cO_X(D))^G = \dim_k H^0\left(Y, \pi_*^G(\cO_X(D))\right)}\\
  & = & \dim_k H^0\left(Y,\cO_Y\left(\left\lfloor \frac{\pi_*(D)}{n}\right\rfloor \right)\right)\\
  & = & 1-g_Y + \deg \left\lfloor \frac{\pi_*(D)}{n}\right\rfloor \\
  &=& 1-g_Y + \sum_{Q \in Y} \left\lfloor \frac{n_Q}{e_Q} \right \rfloor \\
  &=& 1- g_Y + \frac{1}{n} \deg(D) - \sum_{Q \in Y} \left\langle \frac{n_Q}{e_Q}\right\rangle,
  \end{eqnarray*}
  as stated.
  \end{proof}


The following corollary computes the dimension of $H^0(X,\Omega_X^{\otimes m})^G$ if $g_X \ge 2$. (If $g_X =0$ or $g_X =1$, see Example~\ref{example3}.)
In particular we see that this dimension is completely determined by $m$, $g_Y$ and $\deg \left\lfloor \frac{m\pi_*(R)}{n} \right\rfloor$.

  \begin{cor}\label{dim}
    Let $m\geq 1$ and suppose that $g_X \ge 2$. Then we have:
    \[\dim_k H^0(X,\Omega_X^{\otimes m})^G =
		   \begin{cases}
	    g_Y \hspace*{2cm} \mbox{if } m=1 \mbox{ and } \pi \mbox{ is tamely ramified}, \\
\\
	    (2m-1)(g_Y-1) + \deg\left\lfloor\frac{m\pi_*(R)}{n} \right\rfloor \hspace*{1cm} \mbox{otherwise}.
	  \end{cases}\]
   \end{cor}

   \begin{proof}
   If $\pi$ is tamely ramified, then $\delta_P = e_P-1$ for all $P \in X$ and the divisor~$\left\lfloor \frac{\pi_*(R)}{n}\right\rfloor$ is the zero divisor. We therefore have
   \[\left\lfloor \frac{\pi_*(K_X)}{n} \right\rfloor = \left\lfloor \frac{\pi_*(\pi^*(K_Y)) + \pi_*(R)}{n} \right\rfloor =
   \left\lfloor \frac{nK_Y + \pi_*(R)}{n} \right\rfloor = K_Y\]
   and, using Lemma~\ref{folklore}, we obtain
   \[\dim_k H^0(X,\Omega_X)^G = \dim_k H^0\left(Y, \pi_*^G(\cO_X(K_X))\right) = \dim_k H^0(Y, \cO_Y(K_Y)) = g_Y,\]
   as stated. \\
   If $\pi$ is not tamely ramified, then the double sum $\sum_{P \in X} \sum_{j \ge 1} \left(\ord(G_j(P))-1\right)$ is positive. On the other hand, if $m \ge 2$, then we have $m(2g_X-2) > 2g_X-2$ since we have assumed that $g_X \ge 2$. So, in either case we have
   \[\deg(mK_X) = m(2g_X-2) > 2g_X-2-\sum_{P\in X} \sum_{j \ge 1} \left(\ord(G_j(P))-1\right).\]
   We temporarily write $\sum_{P\in X}n_P[P]$ for $K_X$ and, as above, for any $Q \in Y$ and $P \in \pi^{-1}(Q)$, we write $n_Q$ for $n_P$.  Using the previous proposition and Hurwitz formula~(\ref{Hurwitz}) we then obtain
   \begin{eqnarray*}
   \lefteqn{\dim_k H^0(X, \Omega_X^{\otimes m})^G = \dim_k H^0(X, \cO_X(m K_X))^G}\\
   &=& 1-g_Y + \frac{1}{n} (m(2g_X-2)) - \sum_{Q \in Y} \left\langle \frac{m n_Q}{e_Q} \right\rangle\\
   &=& 1-g_Y + m(2g_Y -2) + \frac{m}{n} \deg(R) - \sum_{Q \in Y} \left\langle \frac{m n_Q}{e_Q} \right\rangle \\
   &=& (2m-1)(g_Y-1) + \deg \left\lfloor \frac{m \pi_*(R)}{n} \right\rfloor
   \end{eqnarray*}
   because $\frac{m \pi_*(K_X)}{n} = \frac{m \pi_*(\pi*(K_Y)) + m \pi_*(R)}{n} = m K_Y + \frac{m \pi_*(R)}{n}$ and $\deg(R) = \deg(\pi_*(R))$. This finishes the proof of Corollary~\ref{dim}.
   \end{proof}

If $m=1$ we reformulate Corollary~\ref{dim} in the following slightly more concrete way.
Let $S$ denote the set of all points $Q\in Y$ such that $\pi$ is not tamely ramified above~$Q$, and let $s$ denote the cardinality of $S$.
Note that $s=0$ if $p$ does not divide~$n$.

\begin{cor}\label{dim2}
  We have
    \begin{eqnarray*}
      \dim_k H^0(X,\Omega_X)^G =
	\begin{cases}
	  g_Y & \mbox{if } s=0, \\
	  g_Y-1+\sum_{Q\in S}\left\lfloor \frac{\delta_Q}{e_Q} \right\rfloor & \mbox{otherwise}.
	\end{cases}
    \end{eqnarray*}
\end{cor}

\begin{proof}
  We have
    \[
	\deg\left\lfloor\frac{\pi_*(R)}{n} \right\rfloor = \sum_{Q\in Y}\left\lfloor\sum_{P\mapsto Q} \frac{\delta_P}{n} \right\rfloor = \sum_{Q\in Y} \left\lfloor \frac{\delta_Q}{e_Q} \right\rfloor.
    \]
Furthermore we have $\left\lfloor \frac{\delta_Q}{e_Q} \right\rfloor = 0$ if and only if $\delta_Q<e_Q$, i.e.\ if and only $Q\notin S$.
Thus Corollary \ref{dim2} follows from Corollary~\ref{dim}.
\end{proof}

\begin{rem}
  If $p>0$ and $G$ is cyclic, then Corollary \ref{dim2} can be derived from Proposition~6 in the recent pre-print
 \cite{kako} by Karanikolopoulos and Kontogeorgis.
\end{rem}

  \section{Faithfulness of Actions on the Space of Global \\Holomorphic Differentials}\label{GlobalDifferentials}

  In this section we consider the space $H^0(X,\Omega_X)$ of global holomorphic differentials on $X$ and prove that the action of the group~$G$ on this space is faithful if and only if $G$ does not contain a hyperelliptic involution or if $p\not=2$, see Theorem~\ref{faithful1}. The proof is based on the following criterion for the action of~$G$ on $H^0(X,\Omega_X)$ to be trivial.

  \begin{prop}\label{m=1}
    We assume that $p  > 0$, that $G$ is cyclic of order $p$, that $g_X \ge 2$ and that $g_Y=0$. Then $G$ acts trivially on $H^0(X,\Omega_X)$ if and only if $p=2$.
  \end{prop}
  \begin{proof}
  Let $P_1,\ldots ,P_r \in X$ denote the ramification points of $\pi$. We write $e_i$ and $\delta_i$ for $e_{P_i}$ and $\delta_P{_i}$.
  Also, for $i=1, \ldots, r$, we define $N_i \in \NN$ by $\textrm{ord}_{P_i}(\sigma(t) - t) = N_i +1$ where $t$ is a local parameter at the ramification point $P_i$ and $\sigma$ is a generator of the decomposition group $G_0(P_i)$. From Lemma~1 on p.~87 in \cite{Naka} we know that $p$ does not divide $N_i$ for $i=1, \ldots, r$, a fact we will use several times below. We have $\delta_i =(N_i+1)(p-1)$ by Hilbert's formula~(\ref{Hilbert}). Let $N:= \sum_{i=1}^r N_i$. Using the Hurwitz formula~(\ref{Hurwitz}) we then obtain
%
      \begin{equation}\label{hur2}
	2g_X - 2 = -2p + (N+r)(p-1)
      \end{equation}
    and hence
      \[
	\textrm{dim}_k H^0(X,\Omega_X) = g_X =\frac{(N+r-2)(p-1)}{2}.
      \]
    Since $g_X \ge 0$ we obtain $r \ge 1$; that is, $\pi$ is not unramified.
    As $\textrm{char}(k)=p=\textrm{ord}(G)$, the morphism $\pi$ is thus not tamely ramified and the cardinality~$s$ defined at
    the end of the previous section is not zero. From Corollary~\ref{dim2} we conclude that
    \[\textrm{dim}_k H^0(X, \Omega_X)^G = g_Y-1+\sum_{i=1}^r \left\lfloor \frac{\delta_i}{e_i} \right\rfloor = -1 +N +r +\sum_{i=1}^r \left\lfloor
    - \frac{N_i+1}{p}\right\rfloor.\]

    If $p=2$, the dimensions of $H^0(X,\Omega_X)$ and $H^0(X,\Omega_X)^G$ are therefore equal (to $\frac{N+r-2}{2}$).
    This shows the `if' direction in Proposition~\ref{m=1}.

    To prove the other direction we now assume that $G$ acts trivially $H^0(X, \Omega_X)$ and we suppose that $p \ge 3$. We will show that this contradicts our assumption that $g_X \ge 2$.
    For each $i=1, \ldots, r$, we write $N_i = s_i p +t_i$ with $s_i \in \NN$ and $t_i \in \{1, \ldots, p-1\}$.
    We furthermore put $S:=\sum_{i=1}^r s_i$ and $T:= \sum_{i=1}^r t_i \ge r$.
    Then we have
      \[
	 \frac{(N+r-2)(p-1)}{2} =\textrm{dim}_k H^0(X,\Omega_X)  = \textrm{dim}_k H^0(X,\Omega_X)^G = N-S-1 .
      \]
    Rearranging this equation we obtain
      \[
	 (3-p)N - 2S = (r-2)(p-1) +2
      \]
    and hence
      \[
	 (-p^2 + 3p -2)S = (r-2)(p-1) +2 - (3-p)T.
      \]
    Since $-p^2+3p-2 = - (p-1)(p-2)$ and $p \ge 3$, this equation implies that
      \[
	S = \frac{(r-2)(1-p)-2 + T (3-p)}{(p-1)(p-2)}.
      \]
    Because $S \ge 0$, the numerator of this fraction is non-negative, that is
      \begin{eqnarray*}
	\lefteqn{0 \le (r-2)(1-p) - 2 + T (3-p)}\\
	&\le & (r-2)(1-p) - 2 + r (3-p)\\
	&=& 2 (r-1)(2-p).
      \end{eqnarray*}
    Hence we have $r=1$ and that numerator is $0$.
    We conclude that $S=0$ and that $T=1$ or $p=3$.
    If $T=1$ we also have $N=1$ and finally
      \[
	g_X = \frac{(N+r-2)(p-1)}{2} = 0,
      \]
    a contradiction.
    If $T \not=1$ and $p=3$ we obtain $N=T=2$ and finally
      \[
	g_X = \frac{(N+r-2)(p-1)}{2} =1,
      \]
    again a contradiction.
  \end{proof}

  \begin{thm}\label{faithful1}
    Suppose that $g_X\geq 2$. Then $G$ does not act faithfully on $H^0(X,\Omega_X)$ if and only if $G$ contains a hyperelliptic involution and $p=2$.
   \end{thm}

\begin{rem} Note that the existence of a hyperelliptic involution $\sigma$ in $G$ means that not only the genus of $X/\langle \sigma \rangle$ but also the genus of $Y=X/G$ is $0$ (by the Hurwitz formula~(\ref{Hurwitz})). Again by the Hurwitz formula, the canonical projection $X\rightarrow X/\langle \sigma \rangle$ cannot be unramified. If $p=2$, it can therefore not be tamely ramified and $\pi$ cannot be tamely ramified either. Thus, Theorem~\ref{faithful1} implies that, if the action on $H^0(X,\Omega_X)$ is not faithful, then we also have that $g_Y=0$ and that $\pi$ is not tamely ramified.
\end{rem}

  \begin{proof}
    We first show the `if' direction. The hyperelliptic involution contained in~$G$ generates a subgroup of order $2$.
    Since $p=2$, this acts trivially by Proposition~\ref{m=1}, and hence $G$ does not act faithfully.

    We now assume that $G$ does not act faithfully on $H^0(X,\Omega_X)$.
    By replacing $G$ with the (non-trivial) kernel $H$ if necessary, we may assume that $G$ is non-trivial and acts trivially on $H^0(X,\Omega_X)$.

    We first prove that $\pi$ is not tamely ramified. Suppose that $\pi$ is tamely ramified. Then by Corollary~\ref{dim2} we have:
      \begin{equation*}
	g_X=\dim_k H^0(X,\Omega_X)=\dim_k H^0(X,\Omega_X)^G=g_Y.
      \end{equation*}
    Substituting this into the Hurwitz formula~(\ref{Hurwitz}) yields the desired contradiction because $g_X\geq 2, n\geq 2$ and $\deg(R)\geq 0$.

    As $\pi$ is not tamely ramified, the characteristic~$p$ of $k$ is positive and the group~$G$ has a subgroup of order $p$; by replacing~$G$ with that subgroup we may assume that $G$ is cyclic of order $p$. Now Theorem~\ref{faithful1} will follow from Proposition~\ref{m=1} once we have shown that $g_Y=0$.

    Corollary~\ref{dim2} gives us that
      \begin{equation*}
	g_X=\dim_k H^0(X,\Omega_X)=\dim_k H^0(X,\Omega_X)^G= g_Y-1+\sum_{Q\in S} \left\lfloor \frac{\delta_Q}{p} \right\rfloor
      \end{equation*}
    where $S$ is the set of all points $Q \in Y$ such that $\pi$ is not tamely ramified above $Q$.
    Substituting this in to the Hurwitz formula (\ref{Hurwitz}), we see that
      \begin{equation*}
	2\left(g_Y - 1 + \sum_{Q \in S}\left \lfloor \frac{\delta_Q}{p} \right \rfloor -1 \right) = 2p (g_Y -1) + \textrm{deg}(R).
      \end{equation*}
    Rewriting the previous equation yields
      \begin{eqnarray*}
	\lefteqn{(2p-2)g_Y = 2p-4 + 2 \sum_{Q \in S} \left\lfloor \frac{\delta_Q}{p}\right \rfloor - \textrm{deg}(R)}\\
	&=& 2 \left(p-2 + \sum_{Q \in S} \left(\left\lfloor  \frac{\delta_Q}{p} \right\rfloor -  \frac{\delta_Q}{2}\right) \right)\\
	& \le & 2(p-2).
      \end{eqnarray*}
    Hence we obtain $g_Y \le \frac{p-2}{p-1} < 1$ and therefore $g_Y =0$, as desired.
    \end{proof}

The curves occurring in Theorem~\ref{faithful1} are hyperelliptic curves in characteristic $p=2$. The general standard equation for such curves will be stated in Section~\ref{hyperelliptic}. We give a simple example covering every genus $g_X\ge 2$ already now.

\begin{example}\label{example}
 We suppose that $p=2$. Let $r$ be an odd natural number, let $k(x,y)$ be the extension of the rational function field $k(x)$ given by the Artin-Schreier equation $y^2-y=x^r$ and define $\pi: X \rightarrow \PP^1_k$ to be the corresponding cover of non-singular projective curves over $k$. Then we have  $\textrm{dim}_k H^0(X,\Omega_X) = g_X= \frac{r-1}{2}$ (e.g.\ see~\cite[Example~2.5]{galoisstruc}).
\end{example}

\begin{rem}\label{shortproofs}
(a) The paper~\cite{ValMad} by Valentini and Madan is about determining the $k[G]$-module structure of the space~$H^0(X,\Omega_X)$ if $G$ is a cyclic $p$-group. With some effort it is also possible to derive major steps of this section from their fine results.\\
(b) If $X$ is not hyperelliptic, the following argument yields a very short proof of (the `only-if' direction of) Theorem~\ref{faithful1}. By Proposition~IV.5.2 in \cite{hart} the canonical morphism $X \ra \PP(H^0(X, \Omega_X))$ is a $G$-equivariant closed embedding; as the action of~$G$ on~$X$ is faithful, the action of~$G$ on~$H^0(X, \Omega_X)$ has therefore to be faithful as well. A similar, but more intricate argument based on the deeper Proposition~IV.5.3 in \cite{hart}, can actually be used to prove Theorem~\ref{faithful1} also if $X$ is hyperelliptic.
\end{rem}

\section{Trivial Actions and Faithful Actions on \\Riemann-Roch Spaces}\label{RiemannRochSpaces}

  The goal of this section is to give both sufficient and necessary conditions for the action of~$G$ on $H^0(X,\cO_X(D))$ to be faithful if $\textrm{deg}(D) > 2g_X -2$. For instance, if $m \ge 2$, the group~$G$ does not act faithfully on the space $H^0(X, \Omega_X^{\otimes m})$ of global polydifferentials of order~$m$ if and only if $G$ contains a hyperelliptic involution and $m=g_X=2$, see Corollary~\ref{faithful2}. We begin with a criterion for the action of $G$ on $H^0(X,\cO(D))$ to be trivial.

  \begin{thm}\label{trivialD}
  Let $D=\sum_{P\in X} n_P [P]$ be a $G$-invariant divisor on $X$ such that $\deg(D) > 2g_X-2$. Then the action of $G$ on $H^0(X,\cO_X(D))$ is trivial if and only if
  \begin{equation}\label{trivialDequation}
  (n-1) \deg(D) = n\left(g_X - g_Y - \sum_{Q \in Y} \left\langle \frac{n_Q}{e_Q} \right\rangle \right).
  \end{equation}
  (Recall that $n_Q:=n_P$ for $Q \in Y$ and $P \in \pi^{-1}(Q)$.)
  \end{thm}

  \begin{proof}
  The action of $G$ on $H^0(X, \cO_X(D))$ is trivial if and only if
  \[\dim_k H^0(X,\cO_X(D)) = \dim_k H^0(X, \cO_X(D))^G.\]
  Using the Riemann-Roch formula \cite[Ch.~IV, {\S}1, Theorem~1.3 and Example~1.3.4]{hart} for the left-hand dimension and the formula given by Proposition~\ref{dimD} for the right-hand dimension, we obtain that the action of $G$ on $H^0(X, \cO_X(D))$ is trivial if and only if
  \[1- g_X + \deg(D) = 1-g_Y + \frac{1}{n} \deg(D) - \sum_{Q \in Y} \left\langle \frac{n_Q}{e_Q} \right\rangle.\]
  This condition rearranges to condition~(\ref{trivialDequation}),
  as desired.
  \end{proof}

  \begin{cor}\label{trivialD2}
    Let $D= \sum_{P \in X} n_P [P]$ be a $G$-invariant divisor on $X$.  We assume that $\deg(D)\geq 2g_X$, that $n \ge 2$ and that $g_X \ge 1$. Then the action of the group~$G$ on $H^0(X,\cO_X(D))$ is trivial if and
    only if $\deg(D)=2g_X$, $n=2$, $g_Y=0$ and $n_P$ is even for each ramification point $P \in X$.
  \end{cor}

  \begin{proof}
    The following inequalities always hold under the stated assumptions:
    \[(n-1)\deg(D) \geq (n-1)2g_X \geq n g_X
     \ge n \left(g_X-g_Y-\sum_{Q\in Y}\left\langle\frac{n_Q}{e_Q}\right\rangle\right).\]
    Now the first inequality is an equality if and only if $\deg(D)=2g_X$.
    The second is an equality if and only if $n=2$.
    The third inequality is an equality if and only if $g_Y=0$ and $\sum_{Q\in Y}\left\langle\frac{n_Q}{e_Q}\right\rangle=0$. The latter is the case if and only if each $n_Q$ is divisible by~$e_Q$, which, if $n=2$, means that $n_P$ is even for each ramification point $P \in X$. Given these observations, Theorem~\ref{trivialD} implies Corollary~\ref{trivialD2}.
  \end{proof}


  \begin{cor}\label{trivialPoly}
    Let $m\geq 2$. We assume that $n \ge 2$ and that $g_X \ge 1$. Then the action of $G$ on $H^0(X,\Omega_X^{\otimes m})$ is trivial if and only if $g_Y=0$ and $n=g_X=m=2$   .
  \end{cor}

  \begin{proof}
    As $g_X\geq 2$ and $m\geq 2$ we have that $\deg(mK_X)\geq 2g_X$.
    So, by Corollary~\ref{trivialD2}, the action of~$G$ on $H^0(X,\Omega_X^{\otimes m})$ is trivial if and only if $\deg(mK_X)=2g_X$, $n=2$, $g_Y=0$ and, for each ramification point~$P \in X$, the coefficient of the divisor $mK_X$ at~$P$ is even. Now $\deg(mK_X)=2g_X$ means that $m(2g_X-2)=2g_X$, i.e.\ that $m(g_X-1)=g_X$, and hence that $m=g_X=2$. It therefore suffices to prove that, if $n=2$, the coefficient~$n_P$ of the divisor $K_X=\pi^*(K_Y)+R$ at each ramification point~$P \in X$ is always even. By definition, the coefficient of the pull-back divisor $\pi^*(K_Y)$ at $P$ is even. Furthermore, the coefficient~$\delta_P$ of $R$ at $P$ is even, see the proof of Proposition~\ref{m=1}. Hence also $n_P$ is even.
  \end{proof}

To illustrate the conditions in Corollary~\ref{trivialPoly}, we now give simple examples of hyperelliptic curves of genus~2 and state a basis of the corresponding space of global holomorphic quadratic differentials.

\begin{example} \label{example2}
If $p\not= 2$, let $k(x,y)$ be the extension of the rational function field~$k(x)$ given by $y^2=(x-x_1)\cdots (x-x_6)$, where $x_1, \ldots , x_6\in k$ are pairwise distinct.       Then the corresponding natural projection $\pi:X\rightarrow \mathbb{P}_k^1$ is of degree $2$ and ramified exactly over $x_1,\ldots , x_6\in \mathbb{P}_k^1$. In particular we have $g_X=2$ by formulae~(\ref{Hurwitz}) and~(\ref{Hilbert}). Furthermore, the three quadratic differentials $\frac{\textrm{d}x^{\otimes 2}}{y^2}, \, x \,\frac{\textrm{d}x^{\otimes 2}}{y^2},\, x^2\, \frac{\textrm{d}x^{\otimes 2}}{y^2}$ are obviously fixed by the hyperelliptic involution $y \mapsto -y$ and form a basis of $H^0(X,\Omega_X^{\otimes 2})$ by Theorem~\ref{basispoly} below.
If $p=2$, then the curve~$X$ considered in Example~\ref{example} satisfies $g_X=2$ when $r=5$. Furthermore the quadratic differentials $\textrm{d}x^{\otimes 2}, x \textrm{d}x^{\otimes 2}, x^2 \textrm{d}x^{\otimes 2}$ are obviously fixed by the hyperelliptic involution $y \mapsto y+1$ and form a basis of $H^0(X, \Omega^{\otimes 2}_X)$ by Theorem~\ref{basispoly} below.
\end{example}

    \begin{cor}\label{faithful2}
    Let $m\geq 2$ and suppose that $g_X\geq 2$. Then $G$ does not act faithfully on $H^0(X,\Omega_X^{\otimes m})$ if and only if $G$ contains a hyperelliptic involution and $m=2$ and $g_X=2$.
    \end{cor}

    \begin{proof}
    We first prove the `if' direction. The subgroup of $G$ generated by the hyperelliptic involution is a group of order $2$ acting on $H^0(X,\Omega_X^{\otimes m})$. Since $g_X=m=2$, the action of this subgroup is trivial by Corollary~\ref{trivialPoly}, and this implies that $G$ does not act faithfully.

    To prove the other direction we apply Corollary~\ref{trivialPoly} to the non-trivial kernel of the action of $G$ on $H^0(X,\Omega_X^{\otimes m})$. \end{proof}

In the following examples we look at the cases $g_X=0$ and $g_X=1$ which are not covered by the previous corollary.

\begin{example} \label{example3}
Let $g_X =0$, i.e.\ $X \cong \PP^1_k$. Then the degree of the canonical divisor $K_X$ on $X$ is $-2$ and so $\deg(mK_X)<0$ for all $m\geq1$.
  Hence  $H^0(X,\Omega_X^{\otimes m}) = \{0\}$ by \cite[Ch.~IV, Lemma~1.2]{hart} and every automorphism of~$X$ acts trivially on $H^0(X,\Omega_X^{\otimes m})$ for all $m\geq 1$.
\end{example}

\begin{example}
Let $g_X =1$, i.e.\ $X$ is an elliptic curve. Then the $\mathcal{O}_X$-module $\Omega_X^{\otimes m}$ is free of rank~1 for all $m \ge 1$. Hence $\dim_k H^0(X,\Omega_X^{\otimes m}) =1$ for all $m\ \ge 1$ and the canonical homomorphism $H^0(X, \Omega_X)^{\otimes m} \ra H^0(X, \Omega_X^{\otimes m})$ is bijective. We therefore study the action of $\textrm{Aut}(X)$ on $H^0(X, \Omega_X^{\otimes m})$ only for $m =1$. Let $\chi: \textrm{Aut}(X) \ra k$ denote the corresponding multiplicative character and let $j \in k$ denote the $j$-invariant of $X$. We are going to describe the kernel of $\chi$ and to show that the image of $\chi$ is the group $\mu_r(k)$ of $r^\textrm{th}$ roots of unity in $k$ with $r$ given by the following table.
\begin{center}
\begin{tabular}{c||c|c|c|c|c|c|c}
$p$& $\not= 2,3$ & $\not= 2,3$ & $\not={2,3}$ & 3 & 3 & 2 & 2\\
\hline
$j$ & $\not= 0, 1728$ & 1728 & 0 & $\not= 0$ & 0 & $\not= 0$ & 0 \\
\hline  \hline
$r$ & 2 & 4 & 6 & 2 & 4 & 1 & 3
\end{tabular}
\end{center}

As any basis $\omega$ of $H^0(X, \Omega)$ is translation invariant \cite[Proposition~III.5.1]{Sil}, the normal subgroup~$X(k)$ of~$\textrm{Aut}(k)$ consisting of all translations is contained in the kernel of this action. By \cite[Theorem~III.10.1]{Sil}, the subgroup~$G$ of~$\textrm{Aut}(X)$ consisting of those automorphisms which fix the zero point is finite and the canonical homomorphism from $G$ to the factor group $\textrm{Aut}(X)/X(k)$ is bijective. Let $\bar{\chi}: G \ra k$ denote the induced character. We now distinguish the following cases.\\
(i) Let $p \not=2, 3$. By \cite[Corollary~III.10.2]{Sil}, the group~$G$ is cyclic of order 2, 4 or 6 depending on whether $j\not=0, 1728$, $j=1728$ or $j=0$. Furthermore, $\bar{\chi}$ is injective, i.e.\ the action of~$G$ on~$H^0(X,\Omega_X)$ is faithful. Indeed, given a Weierstrass equation $y^2=x^3+Ax+B$ for $X$, the action of any generator $\sigma$ of $G$ is given by $(x,y) \mapsto (\zeta^2 x, \zeta^3y)$ where $\zeta$ is a primitive root of unity of order 2, 4 or 6, respectively, see the proof of \cite[Corollary~III.10.2]{Sil}. As $\omega = \frac{\textrm{d}x}{y}$ \cite[Section~III.5]{Sil}, we obtain that $\chi(\sigma) = \zeta^{-1}$ and that $\chi$ is injective. \\
(ii) Let $p=3$. If $j\not=0$, then $\textrm{ord}(G) = 2$ \cite[Proposition~A.1.2]{Sil} and, using Case~I in the proof of {\em ibid.}, the same reasoning as in (i) shows that $\bar{\chi}$ is injective. If $j=0$, the group~$G$ is a semidirect product of a normal subgroup~$C_3$ of order 3 and a cyclic subgroup of order 4, see \cite[Exercise~A.1(a)]{Sil}. The character $\bar{\chi}: G \ra k$ is trivial on $C_3$ because $\mu_3(k)$ is trivial. Using Case~II in the proof of {\em ibid.}, the same reasoning as in (i) shows that the induced character $\bar{\bar{\chi}}: C_4 \ra k$ is injective. \\
(iii) Let $p=2$. If $j \not=0$, then $\textrm{ord}(G)=2$ \cite[Proposition~A.1.2]{Sil}. We conclude that $\bar{\chi}$ is trivial because $\mu_2(k)$ is trivial. If $j=0$, the group~$G$ is a semidirect product of a cyclic subgroup~$C_3$ and a normal subgroup~$Q$ isomorphic to the quaternion group of order 8, see \cite[Exercise~A.1(b)]{Sil}. Again, as $\mu_8(k)$ is trivial, the character~$\bar{\chi}$ is trivial on~$Q$. Using Case~IV in the proof of {\em ibid.}, one easily shows that the induced character $\bar{\bar{\chi}}: C_3 \ra k$ is injective. Note that here $\omega = \textrm{d}x$, see \cite[Proposition~A.1.1(c) and Section~III.5]{Sil}.
\end{example}

Similarly to the case $\textrm{deg}(D) \ge 2g_X$ in Corollary~\ref{trivialD2}, the following corollary gives, in the case $\textrm{deg}(D)=2g_X-1$, necessary and sufficient conditions for the action of~$G$ on $H^0(X,\cO_X(D))$ to be trivial.

  \begin{cor}\label{trivialD3}
    Let $D = \sum_{P \in X} n_P [P]$ be $G$-invariant divisor on $X$.
    We assume that $\deg(D)= 2g_X-1$, that $n\ge 2$ and that $g_X \ge 2$. Then the action of $G$ on the space $H^0(X,\cO_X(D))$ is trivial if and only if $g_Y=0$ and one of the following two sets of conditions holds:
      \begin{itemize}
	\item  $n=2$ and there is exactly one ramification point $P \in X$ for which $n_P$ is odd;
	\item  $n=3$, $g_X=2$ and $n_P$ is a multiple of $3$ for each ramification point $P \in X$.
      \end{itemize}
  \end{cor}


  \begin{proof}

     As $\deg(D)=2g_X-1$, we conclude from Theorem~\ref{trivialD} that the action is trivial if and only if
      \begin{equation*}
	(n-1)(2g_X-1)=n\left(g_X-g_Y-\sum_{Q\in Y}\left\langle\frac{n_Q}{e_Q}\right\rangle\right).
      \end{equation*}
    If $n=2$, then this is equivalent to $2g_X-1=2g_X-2g_Y-2\sum_{Q\in Y}\left\langle\frac{n_Q}{e_Q}\right\rangle$ and hence to $g_Y=0$ and $\sum_{Q\in Y}\left\langle\frac{n_Q}{e_Q}\right\rangle=\frac{1}{2}$, and the latter condition means that there is exactly one ramification point $P \in X$ for which $n_P$ is odd.

    If $n\geq 3$, then, as $g_X\geq 2$, we have $g_X\geq \frac{n-1}{n-2}$ which is equivalent to the first inequality in the following chain of inequalities:
      \begin{equation*}
	(n-1)(2g_X-1) \geq n g_X
    \ge n\left(g_X-g_Y-\sum_{Q\in Y} \left\langle\frac{n_Q}{e_Q}\right\rangle\right).
      \end{equation*}
    Hence the action is trivial if and only if both inequalities are equalities, which is the case if and only if $n=3,\ g_X=2$, $g_Y=0$ and $e_Q\mid n_Q$ for all $Q\in Y$. When $n=3$, the latter condition means that $n_P$ is a multiple of $3$ for each ramification point $P \in X$.
  \end{proof}

  Corollaries~\ref{trivialD2} and~\ref{trivialD3} yield the following sufficient conditions for the action of~$G$ on a general Riemann-Roch space~$H^0(X,{\cal O}_X(D))$ to be faithful.

  \begin{cor}\label{trivialD4}
   Let $g_X \ge 2$ and let $D=\sum_{P\in X} n_P [P]$ be a $G$-invariant divisor on~$X$. Let $X_\textnormal{ram} := \{P \in X:\pi \textrm{ is ramified at } P\}$. Then the action of $G$ on $H^0(X,\cO_X(D))$ is faithful if any of the following four sets of conditions holds:
     \begin{enumerate}
     \item[(a)] $\deg(D) \ge 2g_X+1$;
     \item[(b)] $\deg(D) = 2g_X$ and $n_P$ is odd for each $P \in X_\textnormal{ram}$;
     \item[(c)] $\deg(D)=2g_X-1$, $g_X \ge 3$ and $n_P$ is even for each $P \in X_\textnormal{ram}$;
     \item[(d)] $\deg(D)=2g_X-1$, $g_X=2$ and $n_P$ is even but not a multiple of $3$ for each $P\in X_\textnormal{ram}$.
     \end{enumerate}
  \end{cor}

  \begin{proof}
    Suppose the action of $G$ on $H^0(X,\cO_X(D))$ is not faithful. Then there exists a non-trivial subgroup $H$ of $G$ such that the action of $H$ on $H^0(X,\cO_X(D))$ is in fact trivial. \\
    If $\deg(D) \ge 2g_X$, Corollary~\ref{trivialD2} implies that $\deg(D) = 2g_X$, that the order of~$H$ is~$2$, that the genus of $X/H$ is $0$ and that $n_P$ is even for each ramification point~$P$ of the projection $X \ra X/H$. In particular, condition~(a) cannot hold, and condition~(b) cannot hold because $X\ra X/H$ is not unramified (by the Hurwitz~formula~(\ref{Hurwitz})) and because each ramification point of $X \rightarrow X/H$ is also a
    ramification point of~$\pi: X\ra X/G$.\\
    Similarly, if $\deg(D) = 2g_X-1$, Corollary~\ref{trivialD3} implies that none of the conditions~(c) and~(d) can hold. Indeed, each of the conditions~(c) and~(d) contradicts both the first and second set of conditions in Corollary~\ref{trivialD3}. \\
    So we have proved that, if any of the conditions~(a) -- (d) holds, then the action of $G$ on $H^0(X,\cO_X(D))$ is faithful.
  \end{proof}

\begin{rem}
Let $\textrm{deg}(D) \ge 2g_X +1$, which amounts to $g_X\ge 3$ or ($g_X=2$ and $m\ge 3$) in Corollaries~(\ref{trivialPoly}) and~(\ref{faithful2}). Then, as in Remark~(\ref{shortproofs})(b), most of the results of this section are an immediate consequence of the fact that $D$ is very ample, see Corollary~IV.3.2 in \cite{hart}.
\end{rem}

\section{Global Holomorphic Polydifferentials\\ on Hyperelliptic Curves}\label{hyperelliptic}

In this section we assume that the curve~$X$ is hyperelliptic of genus~$g \ge 2$ and give an explicit basis of $H^0(X, \Omega_X^{\otimes m})$ for any $m\ge 1$, see Theorem~\ref{basispoly} below. If furthermore $G$ is the cyclic group of order~$2$ generated by the hyperelliptic involution~$\sigma$, this quickly leads to another proof of Theorem~\ref{faithful1} and Corollary~\ref{faithful2}.

We fix an isomorphism $X/G\cong \PP^1_k$ and consider the projection
\[x: X \rightarrow X/G \cong \PP^1_k\]
as an element of the function field $K(X)$. By Proposition~4.24 and Remark~4.25 in Chapter~7 of~\cite{Liu02}, there exists an element $y \in K(X)$ such that $K(X)=k(x,y)$ and such that $y$ satisfies a quadratic equation over $k(x)$ of the following type:

{\bf Case $p \not=2$}: \hfill  $y^2 = f(x)$  \hspace*{\fill} \\
where $f(x) \in k[x]$ is a polynomial without repeated zeroes.

{\bf Case $p =2$}: \hfill $y^2 - h(x)y = f(x)$ \hspace*{\fill}\\
where $f(x), h(x) \in k[x]$ are non-zero polynomials such that $h'(x)^2f(x) + f'(x)^2$ and $h(x)$ have no common zeroes in~$k$.

We recall that the stated condition on the polynomial(s)~$f(x)$ (and~$h(x)$, respectively) means that the affine plane curve defined by the quadratic equation is smooth, see~\cite[Chap.~7, Remark~4.25]{Liu02}.

Let $m\ge 1$ and let the meromorphic polydifferential $\omega \in \Omega_{K(X)/k}^{\otimes m}$ be defined as follows:
\[\omega := \frac{\textrm{d}x^{\otimes m}}{y^m}  \quad \textrm{ if } p \not= 2 \qquad \textrm{ and } \qquad \omega := \frac{\textrm{d}x^{\otimes m}}{h(x)^m} \quad  \textrm{ if } p =2.\]

\begin{thm}\label{basispoly}
The following polydifferentials form a basis of $H^0(X, \Omega_X^{\otimes m})$:
\[\begin{cases} \omega, x \omega, \ldots, x^{g-1} \omega & \textrm{if } m=1; \\
\omega, x\omega, x^2 \omega & \textrm{if } m=2 \textrm{ and } g=2; \\
\omega, x \omega, \ldots, x^{m(g-1)} \omega;\; y \omega, xy\omega, \ldots, x^{(m-1)(g-1)-2}y\omega & \textrm{otherwise}.
\end{cases}\]
\end{thm}

\begin{rem}
The case~$m=1$ of the previous theorem is for instance also treated in Proposition~4.26 of Chapter~7 in \cite{Liu02}.
\end{rem}

 We now briefly explain that Theorem~\ref{basispoly} yields a new proof of Theorem~\ref{faithful1} and Corollary~\ref{faithful2} if $X$ is hyperelliptic and $G$ is generated by the hyperelliptic involution. By definition, the hyperelliptic involution~$\sigma$ fixes $x$ and maps $y$ to $-y$ if $p \not=2$ and to $y-h(x)$ if $p=2$. We therefore have $\sigma(\omega)= \omega$ if $p=2$ or if $m$ is even. In particular, Theorem~\ref{basispoly} implies that $\sigma$ acts trivially on $H^0(X, \Omega_X^{\otimes m})$ if either $m=1$ and $p=2$ or $m=2$ and $g=2$, as stated in Theorem~\ref{faithful1} and Corollary~\ref{faithful2}. On the other hand, if $p \not= 2$ and $m$ is odd, then $\sigma (x^i \omega) = - x^i \omega$ for $i=0, \ldots, m(g-1)$, so $G$ does act faithfully on~$H^0(X,\Omega_X^{\otimes m})$. Finally, if $m \ge 3$ or $g \ge 3$, the second half of the list of basis elements given in Theorem~\ref{basispoly} is non-empty and $\sigma$ does not act trivially on those basis elements if $p =2$ or if $m$ is even, and so, again, $G$ does act faithfully on~$H^0(X,\Omega_X^{\otimes m})$.

\begin{proof}[Proof (of Theorem~\ref{basispoly})]
 We first observe that the stated family of polydifferentials is linearly independent over~$k$. This follows from the elementary facts that $\omega$ is a basis of the vector space $\Omega_{K(X)/k}$ over $K(X)=k(x,y)$, that $1$ and $y$ are linearly independent over~$k(x)$ and that $1, x, x^2, \ldots$ are linearly independent over~$k$. Furthermore it is easy to see that the number of elements in the stated family is equal to
\[\begin{cases}
g & \textrm{if } m=1\\
(2m-1)(g-1) & \textrm{if } m\ge 2
\end{cases}\]
which in turn is equal to $\dim_k H^0(X,\Omega_X^{\otimes m})$ by the Riemann-Roch theorem (\cite[IV, Theorem~1.3, Examples~1.3.3 and~1.3.4]{hart}). It therefore suffices to prove that each polydifferential in our family is indeed globally holomorphic. \\
For each $a \in \PP^1_k$, let $P_a$ denote the unique point in~$X$ above~$a$, if $a$ is a branch point of~$x$, and let $P_a, P'_a$ denote the two points above~$a$ otherwise.  We write $D_a$ for the divisor
\[D_a=x^*([a]) = \begin{cases}
2[P_a] & \textrm{if } a \textrm{ is a branch point of } x;\\
[P_a] + [P'_a] & \textrm{otherwise}.
\end{cases}\]
Then we obviously have:
\[\textrm{div}(x) = D_0 - D_\infty.\]
Recall that $R$ denotes the ramification divisor of~$x$. By Theorem~3.4.6 of~\cite{Stich09} (which implies the Hurwitz formula~(\ref{Hurwitz})) we have:
\[\textrm{div}(\textrm{d}x) = x^*(\textrm{div}_{\PP_k^1}(\textrm{d}x)) + R = R - 2 D_\infty.\]
We will prove below that
\begin{equation}\label{equality}
\left.\begin{array}{r} \textrm{div}(y) \\ \textrm{div}(h(x)) \end{array} \right\} = R - (g+1) D_\infty \qquad \begin{cases} \textrm{if } p \not=2 \\ \textrm{if } p=2. \end{cases}
\end{equation}
If $p\not= 2$ this equation implies that
\begin{equation}\label{inequality}\textrm{div}(y) \ge -(g+1)D_\infty\end{equation}
and, if $p=2$, we will prove this inequality separately.
For any~$i \ge 0$, we then obtain that
\begin{eqnarray*}
\lefteqn{\textrm{div}(x^i \omega) = \begin{cases} i \,\textrm{div}(x) + m \,\textrm{div} (\textrm{d}x) - m \,\textrm{div}(y) & \textrm{if } p\not= 2\\
i \,\textrm{div}(x) + m \,\textrm{div} (\textrm{d}x) - m \,\textrm{div}(h(x)) & \textrm{if } p= 2 \end{cases}}\\
&=& i(D_0 - D_\infty) + m(R-2D_\infty) - m(R -(g+1) D_\infty)\\
&=& i D_0 + (m(g-1) - i) D_\infty
\end{eqnarray*}
and hence that
\begin{eqnarray*}
\lefteqn{\textrm{div}(x^iy\omega) = \textrm{div}(x^i\omega) + \textrm{div}(y)}\\
&\ge&i D_0 + (m(g-1)-i)D_\infty- (g+1) D_\infty\\
&=& i D_0+ ((m-1)(g-1) -2-i) D_\infty.
\end{eqnarray*}
Thus $x^i \omega$ is holomorphic for $i=0, \ldots, m(g-1)$, and $x^iy\omega$ is holomorphic for $i=0, \ldots, (m-1)(g-1)-2$, as was to be shown.\\
We now prove statements~(\ref{equality}) and~(\ref{inequality}). We first consider the case~$p\not= 2$. Then the degree of~$f(x)$ is equal to $2g+1$ or $2g+2$ by \cite[Chap.~7, Prop.~4.24(a)]{Liu02}. Let $a_1, \ldots, a_{\textrm{deg}(f(x))} \in k$ be the zeroes of $f(x)$. By formulae~(\ref{Hurwitz}) and~(\ref{Hilbert}) we have
\[R=[P_1]+ \ldots + [P_{2g+2}]\]
where $P_i := P_{a_i}$ for $i=1, \ldots, \textrm{deg}(f(x))$ and $P_{2g+2} := P_\infty$ if $\textrm{deg}(f(x)) = 2g+1$.
We then obtain that
\begin{eqnarray*}
\lefteqn{\textrm{div}(y) = \frac{1}{2} \textrm{div} (y^2) = \frac{1}{2} \textrm{div} (f(x))}\\
& =& \begin{cases} [P_1] + \ldots + [P_{2g+2}] - (g+1) D_\infty & \textrm{if } \textrm{deg}(f(x)) = 2g+2;\\
                   [P_1]+ \ldots + [P_{2g+1}] - (2g+1)[P_{\infty}] & \textrm{if } \textrm{deg}(f(x)) = 2g+1.
     \end{cases}\\
&=& R - (g+1) D_\infty
\end{eqnarray*}
which proves both statements~(\ref{equality}) and~(\ref{inequality}) in the case~$p \not=2$.\\
We finally turn to the case~$p=2$. We write $h(x) = \prod_{i=1}^k (x-a_i)^{m_i}$ with $m_1, \ldots, m_k \in \NN$ and pairwise distinct $a_1, \ldots, a_k \in k$. Then $a_1, \ldots, a_k$ are the only branch points of~$x$ in $\mathbb{A}_k^1$ and we let $P_i:= P_{a_i}$ for $i=1, \ldots, k$. Furthermore, let $d:=\textrm{deg}(h(x))=\sum_{i=1}^k m_i$ and $b_i := y(P_i)$ for $i=1, \ldots, k$. By the Nakayama Lemma, $y-b_i$ is a local parameter at~$P_i$. By Hilbert's formula~(\ref{Hilbert}) we then obtain
\[\delta_{P_i} = \textrm{ord}_{P_i}\left(\sigma(y-b_i) - (y-b_i)\right) = \textrm{ord}_{P_i}(-h(x)) = 2m_i\]
for $i=1, \ldots, k$. We hence have
\begin{equation}\label{ramificationdivisor} R=\sum_{i=1}^k 2m_i[P_i] + \left(g+1-d\right) D_\infty\end{equation}
because $\textrm{deg}(R) = 2g+2$ by the Hurwitz formula~(\ref{Hurwitz}). We therefore obtain
\[\textrm{div}(h(x)) = \sum_{i=1}^k 2m_i [P_i] - d \,D_\infty = R - (g+1) D_\infty.\]
This proves equality~(\ref{equality}) in the case~$p=2$. \\
We finally prove inequality~(\ref{inequality}) by contradiction. We first note that $\deg(f(x)) \le 2g+2$ by \cite[Chap.~7, Prop.~4.24(a)]{Liu02}. If $\infty$ is a branch point of~$x$, then we have $d < g+1$ by formula~(\ref{ramificationdivisor}). Now, supposing that inequality~(\ref{inequality}) does not hold implies that $\ord_{P_\infty}(y) < - 2(g+1)$ (which is less than $-2d = \ord_{P_\infty}(h(x))$) and hence that
\[-4(g+1) > 2\ord_{P_\infty}(y) = \ord_{P_\infty} (y(y - h(x))) = \ord_{P_\infty}(f(x))\ge - 2(2g+2)\]
which is a contradiction. If $\infty$ is not a branch point of~$x$, we have $\textrm{deg}(h(x)) = g+1$ by formula~(\ref{ramificationdivisor}). Now, supposing that inequality~(\ref{inequality}) does not hold means that $\ord_P(y) < - (g+1)$ (which is equal to $\ord_P(h(x))$) for $P=P_\infty$ or $P=P'_\infty$ and hence that
\[-2(g+1) > 2\ord_P(y) = \ord_P(y(y-h(x)))  = \ord_P(f(x)) \ge -(2g+2)\]
which again is a contradiction.\\
This concludes the proof of Theorem~\ref{basispoly}.
\end{proof}

\section{Automorphism Groups of Geometric Goppa Codes}\label{Goppa}

Permutation automorphism groups of Goppa codes play an important role in Coding Theory (e.g. see~\cite{Stich09}, \cite{JK} or \cite{Giul08} and the literature cited there). In this section we are going to explain how Corollary~\ref{trivialD4} can be used to obtain permutation groups that act faithfully on geometric Goppa codes. A slightly more explicit account of the basic idea can also be found in Chapter~3 of \cite{FW}.

Let $\cal X$ be a geometrically connected, smooth, projective curve over a finite field~$\FF_q$. Let $D = \sum_{P \in {\cal X} \textrm{ closed}} n_P [P]$ be a divisor on $\cal X$ and let $E$ be a set of $\FF_q$-rational points on $\cal X$ none of which belongs to the support of $D$. Then we have a natural evaluation map
\[\textrm{ev}_{D,E}: H^0({\cal X},\cO_{\cal X}(D)) \rightarrow \textrm{Maps}(E,\FF_q)\]
the image of which is called a {\em geometric Goppa code} and denoted by $C=C(D,E)$. Note that the target space of $\textrm{ev}_{D,E}$ is usually denoted by $\FF_q^r$ where $r$ is the number of points in $E$. Our notation~$\textrm{Maps}(E,\FF_q)$ simplifies the discussions below. 

The group~$\textrm{Sym}(E)$ of permutations of~$E$ acts on $\textrm{Maps}(E, \FF_q)$. The subgroup of $\textrm{Sym}(E)$ consisting of those  $\sigma \in \textrm{Sym}(E)$ that induce an automorphism of~$C$ is called the {\em permutation automorphism group of } $C$ and denoted by $\Aut_\textrm{Perm}(C)$. Note that $\Aut_\textrm{Perm}(C)$ acts on $C$, but not necessarily faithfully.

Now we furthermore assume that $G$ is a finite subgroup of $\Aut({\cal X}/\FF_q)$, that the divisor~$D$ is $G$-invariant and that $\sigma(E)=E$ for all $\sigma \in G$. Then $G$ acts on both the source and target of the evaluation map $\textrm{ev}_{D,E}$ and $\textrm{ev}_{D,E}$ is $G$-equivariant. In particular we have the following composition of obvious group homomorphisms:
\[ G \rightarrow \Aut_\textrm{Perm}(C) \rightarrow \Aut_{\FF_q}(C).\]

\begin{lem}\label{evaluation}
If the cardinality $|E|$ of $E$ is bigger than $\deg(D)$ and $G$ acts faithfully on $H^0({\cal X}, \cO_{\cal X}(D))$, then this composition is injective.
\end{lem}

\begin{proof}
If $|E| > \deg(D)$, then the evaluation map $\textrm{ev}_{D,E}$ is injective by \cite[Corollary~2.2.3]{Stich09} and we have the following obvious commutative diagram:
\[\xymatrix{ G \ar[r] \ar[d] & \Aut_\textrm{Perm}(C) \ar[d]\\
\Aut_{\FF_q}(H^0({\cal X}, \cO_{\cal X}(D))) \ar[r]^{\phantom{xxxxxxx}\sim} & \Aut_{\FF_q}(C).}
\]
Now Lemma~\ref{evaluation} is obvious.
\end{proof}

If $|E| > \deg(D)$ and $G$ acts faithfully on $H^0({\cal X}, \cO_{\cal X}(D))$, then Lemma~\ref{evaluation} allows us to view $G$ as a subgroup of both $\Aut_\textrm{Perm}(C)$ and of $\Aut_{\FF_q}(C)$. Furthermore, when applied to the curve $X={\cal X} \times_{\FF_q} \bar{\FF}_q$ over the algebraic closure $\bar{\FF}_q$ of $\FF_q$, Corollary~\ref{trivialD4} gives us sufficient conditions for the action of $G$ on $H^0({ X}, \cO_{X}(D)) = H^0({\cal X}, \cO_{\cal X}(D)) \otimes_{\FF_q} \bar{\FF}_q$ to be faithful. (Note that here, by abuse of notation, $D$ also denotes the divisor on $X$ induced by the divisor $D$ on $\cal X$.) Under the assumptions of Corollary~\ref{trivialD4} and of Lemma~\ref{evaluation} we thus obtain that $G$ is a subgroup of $\Aut_\textrm{Perm}(C)$  that acts faithfully on the Goppa code $C$. This strengthens Proposition~8.2.3 in \cite{Stich09} in the case $\textrm{deg}(D) \in \{2g_X-1, 2g_X, 2g_X+1\}$ and $g_X \ge 2$. A related result can be found in \cite{JK}.

\section{Computing the Dimension of the Tangent Space of the Equivariant Deformation Functor}

This section depends only on Section~\ref{DimensionFormulae}.

The equivariant deformation problem associated with $(G,X)$ is to determine in how many ways $X$ can be deformed to another curve that also allows $G$ as a group of automorphisms. In \cite{BM00}, Bertin and M\'ezard have shown that the tangent space of the corresponding deformation functor is isomorphic to the equivariant cohomology $H^1(G, {\cal T}_X)$ of $(G,X)$ with values in the tangent sheaf ${\cal T}_X = \Omega^{\vee}_X$. In this section, we apply Corollary~\ref{dim} to prove the following formula for the dimension of $H^1(G,{\cal T}_X)$, provided the space~$M^G$ of invariants and the space~$M_G$ of coinvariants have the same dimension for every finitely generated $k[G]$-module~$M$.

\begin{thm}\label{deformation}
Let $g_X \ge 2$. If $\dim_k M^G = \dim_k M_G$ for every finitely generated $k[G]$-module~$M$, then we have
\begin{equation}\label{DimensionDeformationSpace}
\dim_k H^1(G, {\cal T}_X) = 3g_Y-3 + \sum_{Q \in Y} \left\lfloor \frac{2\delta_Q}{e_Q}\right\rfloor.
\end{equation}
\end{thm}

The following lemma implies that the assumption of the previous theorem is satisfied if $G$ is cyclic and its order is a power of~$p$. In particular, Theorem~\ref{deformation} generalizes Corollary~2.3 in \cite{quaddiffequi} which proves formula~(\ref{DimensionDeformationSpace}) under the assumption that $G$ is cyclic and its order is a power of $p$. Moreover, the proof of Theorem~\ref{deformation} at the end of this section considerably simplifies the proof of Corollary~2.3 in \cite{quaddiffequi} which ultimately relies on a comparatively fine and deep theorem in the last section of Borne's paper~\cite{cohogsheaves}.

\begin{lem}\label{groups}
Suppose that the finite group~$G$ has a normal subgroup~$N$ such that $p$ does not divide the order of~$N$ and such that $G/N$ is cyclic. Then we have $\dim_k M^G = \dim_k M_G$ for every finitely generated $k[G]$-module~$M$.
\end{lem}

\begin{proof} By replacing $N$ with the preimage of the non-$p$-part of the cyclic group $G/N$ under the canonical projection $G \ra G/N$, we may assume that the order of $G/N$ is a power of $p= \textrm{char}(k)$.
We need to show that $\dim_k (M^N)^{G/N} = \dim_k (M_N)_{G/N}$ for every finitely generated $k[G]$-module~$M$. As $p$ does not divide the order of~$N$, the canonical map $M^N \rightarrow M_N$ is obviously an isomorphism of $k[G/N]$-modules. We may therefore assume that $G$ is cyclic and that the order of~$G$ is a power of~$p$. Then, both $\dim_k M^G$ and $\dim_k M_G$ are equal to the number of summands in a representation of $M$ as a direct sum of indecomposable $k[G]$-modules, as one can easily see from the explicit description of indecomposable $k[G]$-modules as given for example in the second paragraph of Section~2 in~\cite{quaddiffequi}.
\end{proof}

Note that the Schur-Zassenhaus theorem tells us that, under the assumptions of Lemma~\ref{groups}, the group~$G$ is in fact a semidirect product of~$N$ and~$G/N$ provided we assume without loss of generality that the order of~$G/N$ is a power of~$p$. Examples of such semidirect products may be obtained as follows. Suppose $q$ is a prime number such that $p$ divides $q-1$ and let $H$ be a (cyclic) subgroup of $(\ZZ/q\ZZ)^\times$ whose order is a power of~$p$. Then $H$ acts on $\ZZ/q\ZZ$ by multiplication, and the semidirect product $H\ltimes \ZZ/q\ZZ$ is of the considered type.

The following simple example shows that the assumption of Theorem~\ref{deformation} cannot be expected to hold true if $G$ is a non-cyclic group whose order is a power of~$p$.

\begin{example}
Let $G$ be the finite group $\ZZ/p\ZZ \times \ZZ/p\ZZ$, represented as the matrix group
\[\left(\begin{array}{ccc} 1 & \ZZ/p\ZZ & \ZZ/p\ZZ \\ 0 & 1 & 0 \\ 0& 0 & 1 \end{array}\right),\]
and let $M$ be the standard representation~$k^3$ of~$G$. Then one easily checks that both $M^G$ and the kernel of the canonical map $M \rightarrow M_G$ are generated by the first standard basis vector of~$k^3$, so $\dim_k M^G =1$ but $\dim_k M_G =2$.
\end{example}

The following lemma will be used in the proof of Theorem~\ref{deformation}. It generalizes and simplifies the considerations in Section~2 of \cite{Konto07}. We use the notation ${}^*$ for the $k$-dual of a vector space over~$k$ or of a $k$-representation of $G$.

\begin{lem}\label{dual}
Let $G$ be a finite group and let $M$ be a finitely generated $k[G]$-module. Then we have a canonical isomorphism
\[(M_G)^\ast \,\, \stackrel{\sim}{\longrightarrow} \,\, (M^*)^G.\]
\end{lem}

\begin{proof}
The dual of the canonical projection~$M \rightarrow M_G$ induces a natural map $\alpha_M: (M_G)^\ast \rightarrow (M^\ast)^G$. Given a representation
\[ k[G]^s \rightarrow k[G]^r \rightarrow M \rightarrow 0\]
of~$M$, we obtain the following commutative diagram with exact rows:
\[\xymatrix{0 \ar[r] & (M_G)^* \ar[r] \ar[d]^{\alpha_M} & \left(\left(k[G]^r\right)_G\right)^* \ar[r]\ar[d]^{\alpha_{k[G]^r}} & \left(\left(k[G]^s\right)_G\right)^* \ar[d]^{\alpha_{k[G]^s}}\\
0 \ar[r] & (M^*)^G \ar[r] & \left(\left(k[G]^r\right)^*\right)^G \ar[r] & \left(\left(k[G]^s\right)^*\right)^G.
}\]
It therefore suffices to proof Lemma~\ref{dual} for $M=k[G]$ in which case it is easy to check.
\end{proof}

\begin{proof}[Proof of Theorem~\ref{deformation}]
A simple spectral-sequence argument (see Proposition~3.1 in \cite{Konto07}) shows that
\[H^1(G, {\cal T}_X) \cong H^1(X, {\cal T}_X)^G.\]
We therefore obtain:
\begin{eqnarray*}
\lefteqn{\dim_k H^1(G, {\cal T}_X) = \dim_k H^1(X, {\cal T}_X)^G}\\
& = & \dim_k (H^0(X, \Omega_X^{\otimes 2})^*)^G \qquad \textrm{(by Serre duality, see \cite[III, 7.12.1]{hart})} \\
&=& \dim_k (H^0(X, \Omega_X^{\otimes 2})_G)^* \qquad \textrm{(by Lemma~\ref{dual})}\\
&=& \dim_k H^0(X, \Omega_X^{\otimes 2})_G\\
&=& \dim_k H^0(X, \Omega_X^{\otimes 2})^G \qquad \qquad \textrm{(by assumption)}\\
&=& 3(g_Y -1) + \textrm{deg} \left\lfloor \frac{2 \pi_*(R)}{n} \right\rfloor \qquad \textrm{(by Corollary~\ref{dim})}\\
&=& 3 g_Y -3 + \sum_{Q \in Y} \left\lfloor \frac{2\delta_Q}{e_Q} \right\rfloor,
\end{eqnarray*}
as was to be shown.
\end{proof}

\section{When does an Automorphism of a Riemann Surface Act Trivially on its First Homology?}\label{homology}

Let $X$ be a connected compact Riemann surface of genus $g \ge 2$, let $m\ge 2$ and let $\sigma$ be an automorphism of $X$ of order $n \not=1$. Rather than the action of $\sigma$ on $H^0(X,\Omega_X^{\otimes m})$, we now study the action of $\sigma$ on the first homology group $H_1(X, \ZZ/m\ZZ)$ of $X$ with values in $\ZZ/m\ZZ$. The object of this section is to point out a striking analogy between these two actions being trivial.

We recall that Corollary~\ref{trivialPoly} states that (in fact for any connected smooth projective curve~$X$ of genus at least 2 over any algebraically closed field) the automorphism~$\sigma$ acts trivially on $H^0(X,\Omega_X^{\otimes m})$ if and only if $m=g_X=2$ and $\sigma$ is a hyperelliptic involution. The following theorem addresses the analogue of the `only-if' direction of this statement.

\begin{thm}\label{trivialhomology}
If $\sigma$ acts trivially on $H_1(X,\ZZ/m\ZZ)$, then $m=2$ and $\sigma$ is an involution.
\end{thm}

\begin{proof}
This follows from the theorem at the end of Section~V.3.4 in \cite{FarkasKra}. We remark that the proof of that theorem is based on a well-known fact (deduced by Serre) about torsion in principal congruence subgroups.
\end{proof}


The next theorem is about the analogue of the `if' direction of Corollary~\ref{trivialPoly}.

\begin{thm}\label{trivialhomology2}
Let $\sigma$ be an involution. Then the implications (a) $\Leftrightarrow$ (b) $\Rightarrow$~(c)~$\Leftrightarrow$~(d) hold for the following statements. 
\begin{enumerate}
\item[(a)] $g=2$ and $\sigma$ is a hyperelliptic involution.
\item[(b)] For every simple closed curve $\alpha$ on $X$, the curve $\sigma(\alpha)$ is freely homotopic to~$\alpha$ or $-\alpha$.
\item[(c)] There exists a basis $B$ of $H_1(X,\ZZ)$ such that $\sigma(x) = \pm x$ for all $x \in B$.
\item[(d)] The involution $\sigma$ acts trivially on $H_1(X,\ZZ/2\ZZ)$.
\end{enumerate}
\end{thm}

\begin{proof}
The equivalence (a) $\Leftrightarrow$ (b) follows from Theorem~1 and Theorem~2 in the paper \cite{HS} by Haas and Susskind and from the fact that any two biholomorphic automorphisms of~$X$ that are homotopic to each other are in fact equal, see \cite[Corollary~2]{Lew}.\\
The implication (b) $\Rightarrow$ (c) follows from the well-known fact that there exists a basis~$B$ of $H_1(X, \ZZ)$ consisting of classes of simple closed curves. It also follows from Theorem~\ref{minusone} below. \\
The implication (c) $\Rightarrow$ (d) is trivial because $H_1(X, \ZZ/2\ZZ) \cong H_1(X,\ZZ) \otimes \ZZ/2\ZZ$. To prove the converse (d) $\Rightarrow$ (c), we observe that for any $x \in H_1(X, \ZZ)$, the classes of $x +\sigma(x)$ and $x-\sigma(x)$ in $H_1(X, \ZZ/2\ZZ)$ are zero; hence $x_+:= \frac{x+\sigma(x)}{2}$ and $x_- := \frac{x-\sigma(x)}{2}$ are well-defined elements in $H_1(X, \ZZ)$ such that $\sigma(x_\pm) = \pm x_\pm$ and $x = x_+ + x_-$. The union of bases for $E_\pm(\sigma):= \{ x\in H_1(X, \ZZ): \sigma(x) = \pm x\}$ is therefore a basis~$B$ of $H_1(X, \ZZ)$ with the required property.
\end{proof}

The following final theorem shows that after dropping the assumption $g=2$ in statement~(a) of the previous theorem, the implication (a) $\Rightarrow$ (c) still holds. In contrast to Corollary~\ref{trivialPoly}, the implication (d) $\Rightarrow$ (a) is therefore not true.

\begin{thm}\label{minusone}
If $\sigma$ is a hyperelliptic involution, then $\sigma$ acts by multiplication with~$-1$ on $H_1(X,\ZZ)$.
\end{thm}

\begin{proof}
Topologically, the hyperelliptic involution~$\sigma$ `rotates $X$ by $180^\circ$ around an axis~$L$' as depicted in Figure~\ref{RiemannSurface}.
\begin{figure}
   \centering
 \includegraphics[width=5in,height=2.8in]{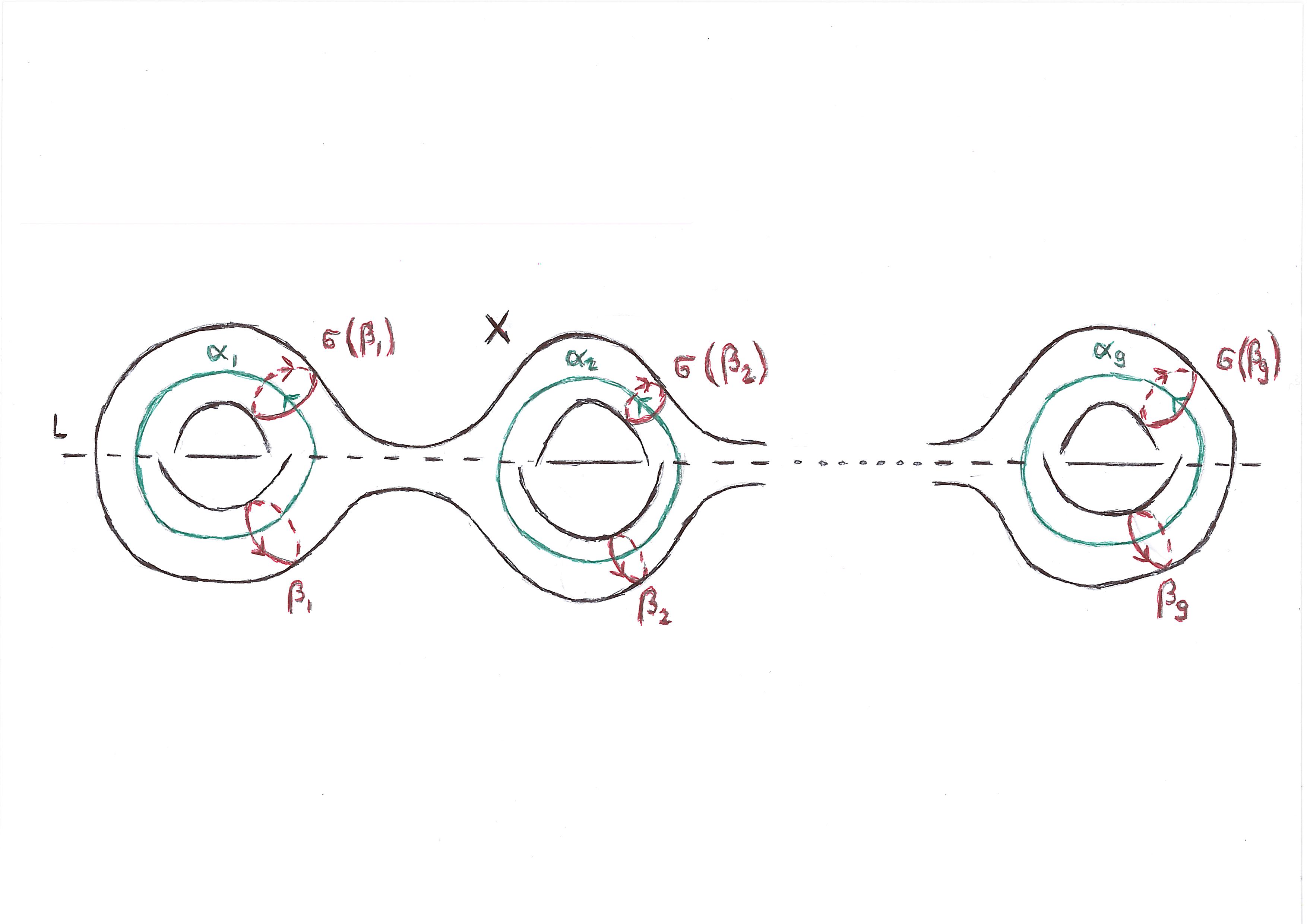}
  \caption{}
  \label{RiemannSurface}
\end{figure}
Let $\alpha_1, \ldots, \alpha_g, \beta_1, \ldots, \beta_g$ be the standard basis elements of $H_1(X, \ZZ)$ as given in Figure~\ref{RiemannSurface}. Then we obviously have $\sigma(\alpha_i) = - \alpha_i$ in $H_1(X,\ZZ)$ for all $i=1, \ldots, g$. Furthermore $\sigma(\beta_1)$ and $\beta_1$ and also $\sigma(\beta_g)$ and $\beta_g$ are homotopic to each other (but with different orientation); hence we have $\sigma(\beta_1)=-\beta_1$ and $\sigma(\beta_g)=-\beta_g$ in $H_1(X,\ZZ)$. To see that $\sigma(\beta_i) = - \beta_i$ also for $i=2, \ldots, g-1$, let $X_i$ be the `left-hand (or right-hand) part of the surface~$X$ bounded by $\beta_i \cup \sigma(\beta_i)$'. Being the oriented boundary of the oriented surface $X_i$ the class $\beta_i + \sigma(\beta_i)$ vanishes in the homology~$H_1(X_i,\ZZ)$ of the subspace~$X_i$ of~$X$ and hence also in $H_1(X,\ZZ)$, as was to be shown.
\end{proof}

We end with the following problem.

{\bf Problem.} Give a geometric characterization of those involutions $\sigma \in \textrm{Aut}(X)$ for which condition~(c) of Theorem~\ref{trivialhomology2} holds.

{\em Acknowledgements.} The authors would like to thank Niels Borne, Allen Broughton, Frank Herrlich, Gareth Jones, Aristides Kontogeorgis, Ian Leary, Michel Matignon and David Singerman for raising various questions underlying this paper and/or for explaining various concepts and ideas concerning particularly the final section. Furthermore the authors would like to thank the referees for carefully reading the paper, for suggesting numerous helpful improvements and for drawing our attention to related work in the literature.

\bibliography{biblio}

\newcommand{\etalchar}[1]{$^{#1}$}
\begin{thebibliography}{FGM{\etalchar{+}}}

\bibitem[BM]{BM00}
J.~Bertin and A.~M{\'e}zard.
\newblock D\'eformations formelles des rev\^etements sauvagement ramifi\'es de
  courbes alg\'ebriques.
\newblock {\em Invent. Math.}, 141(1):195--238, 2000.

\bibitem[Bor]{cohogsheaves}
N.~Borne.
\newblock Cohomology of {$G$}-sheaves in positive characteristic.
\newblock {\em Adv.\ Math.}, 201(2):454--515, 2006.

\bibitem[Bro]{Bro}
S.~A. Broughton.
\newblock The homology and higher representations of the automorphism group of
  a {R}iemann surface.
\newblock {\em Trans. Amer. Math. Soc.}, 300(1):153--158, 1987.

\bibitem[CW]{CW}
C.~Chevalley and A.~Weil.
\newblock {\"Uber das Verhalten der Integrale 1. Gattung bei Automorphismen des
  Funktionenk\"orpers.}
\newblock {\em Abh. Math. Semin. Hamb. Univ.}, 10:358--361, 1934.

\bibitem[FGM{\etalchar{+}}]{FGMPW}
H.~Friedlander, D.~Garton, Beth Malmskog, R.~Pries, and C.~Weir.
\newblock The {$a$}-numbers of {J}acobians of {S}uzuki curves.
\newblock {\em Proc. Amer. Math. Soc.}, 141(9):3019--3028, 2013.

\bibitem[FK]{FarkasKra}
H.~M. Farkas and I.~Kra.
\newblock {\em Riemann surfaces}, volume~71 of {\em Graduate Texts in
  Mathematics}.
\newblock Springer-Verlag, New York, 1980.

\bibitem[FW]{FW}
H.~B. Fischbacher-Weitz.
\newblock Equivariant {R}iemann-{R}och theorems for curves over perfect fields.
\newblock PhD Thesis, University of Southampton, 2008.

\bibitem[FWK]{FWK}
H.~Fischbacher-Weitz and B.~K{\"o}ck.
\newblock Equivariant {R}iemann-{R}och theorems for curves over perfect fields.
\newblock {\em Manuscripta Math.}, 128(1):89--105, 2009.

\bibitem[GJK]{GJK}
D.~Glass, D.~Joyner, and A.~Ksir.
\newblock Codes from {R}iemann-{R}och spaces for {$y^2=x^p-x$} over {${\rm
  GF}(p)$}.
\newblock {\em Int. J. Inf. Coding Theory}, 1(3):298--312, 2010.

\bibitem[GK]{Giul08}
M.~Giulietti and G.~Korchm{\'a}ros.
\newblock On automorphism groups of certain {G}oppa codes.
\newblock {\em Des. Codes Cryptogr.}, 47(1-3):177--190, 2008.

\bibitem[Har]{hart}
R.~Hartshorne.
\newblock {\em Algebraic geometry}, volume~52 of {\em Graduate Texts in
  Mathematics}.
\newblock Springer-Verlag, New York, 1977.

\bibitem[Hor]{Hor}
R.~Hortsch.
\newblock On the canonical representation of curves in positive characteristic.
\newblock {\em New York J. Math.}, 18:911--924, 2012.

\bibitem[HS]{HS}
A.~Haas and P.~Susskind.
\newblock The geometry of the hyperelliptic involution in genus two.
\newblock {\em Proc. Amer. Math. Soc.}, 105(1):159--165, 1989.

\bibitem[JK]{JK}
D.~Joyner and A.~Ksir.
\newblock Automorphism groups of some {AG} codes.
\newblock {\em IEEE Trans. Inform. Theory}, 52(7):3325--3329, 2006.

\bibitem[Kan]{Kani86}
E.~Kani.
\newblock The {G}alois-module structure of the space of holomorphic
  differentials of a curve.
\newblock {\em J. Reine Angew. Math.}, 367:187--206, 1986.

\bibitem[Kar]{Karan12}
S.~Karanikolopoulos.
\newblock {On holomorphic polydifferentials in positive characteristic.}
\newblock {\em Math. Nachr.}, 285(7):852--877, 2012.

\bibitem[KaKo]{kako}
S.~Karanikolopoulos and A.~Kontogeorgis.
\newblock Representation of cyclic groups in positive characteristic and
  {W}eierstrass semigroups.
\newblock {\em J. Number Theory}, 133(1):158--175, 2013.

\bibitem[K\"oKo]{quaddiffequi}
B.~K{\"o}ck and A.~Kontogeorgis.
\newblock Quadratic differentials and equivariant deformation theory of curves.
\newblock {\em Ann.\ Inst.\ Fourier (Grenoble)}, 62(3):1015--1043, 2012.

\bibitem[K{\"o}c]{galoisstruc}
B.~K{\"o}ck.
\newblock Galois structure of {Z}ariski cohomology for weakly ramified covers
  of curves.
\newblock {\em Amer.\ J.\ Math.}, 126(5):1085--1107, 2004.

\bibitem[Kon]{Konto07}
A.~Kontogeorgis.
\newblock Polydifferentials and the deformation functor of curves with
  automorphisms.
\newblock {\em J. Pure Appl. Algebra}, 210(2):551--558, 2007.

\bibitem[Lew]{Lew}
J.~Lewittes.
\newblock Automorphisms of compact {R}iemann surfaces.
\newblock {\em Amer. J. Math.}, 85:734--752, 1963.

\bibitem[Liu]{Liu02}
Q.~Liu.
\newblock {\em Algebraic geometry and arithmetic curves}, volume~6 of {\em
  Oxford Graduate Texts in Mathematics}.
\newblock Oxford University Press, Oxford, 2002.
\newblock Translated from the French by Reinie Ern{\'e}, Oxford Science
  Publications.

\bibitem[Nak1]{Naka}
S.~Nakajima.
\newblock Action of an automorphism of order {$p$} on cohomology groups of an
  algebraic curve.
\newblock {\em J.\ Pure Appl.\ Algebra}, 42(1):85--94, 1986.

\bibitem[Nak2]{Nak2}
S.~Nakajima.
\newblock Galois module structure of cohomology groups for tamely ramified
  coverings of algebraic varieties.
\newblock {\em J. Number Theory}, 22(1):115--123, 1986.

\bibitem[Ser]{localfields}
J.-P. Serre.
\newblock {\em Local fields. {\rm Translated from the French by Marvin Jay
  Greenberg}}, volume~67 of {\em Graduate Texts in Mathematics}.
\newblock Springer-Verlag, New York, 1979.

\bibitem[Sil]{Sil}
J.~H. Silverman.
\newblock {\em The arithmetic of elliptic curves}, volume 106 of {\em Graduate
  Texts in Mathematics}.
\newblock Springer, Dordrecht, second edition, 2009.

\bibitem[Sti]{Stich09}
H.~Stichtenoth.
\newblock {\em Algebraic function fields and codes}, volume 254 of {\em
  Graduate Texts in Mathematics}.
\newblock Springer-Verlag, Berlin, second edition, 2009.

\bibitem[VM]{ValMad}
R.~C. Valentini and M.~L. Madan.
\newblock Automorphisms and holomorphic differentials in characteristic {$p$}.
\newblock {\em J. Number Theory}, 13(1):106--115, 1981.

\end{thebibliography}

\end{document}